\documentclass[a4paper,11pt]{amsart}
\usepackage{amsmath,amsthm,amssymb,enumerate}
\usepackage{enumerate,mathrsfs}

\newtheorem{theorem}{Theorem}[section]
\newtheorem{proposition}[theorem]{Proposition}
\newtheorem{lemma}[theorem]{Lemma}
\newtheorem{corollary}[theorem]{Corollary}

\newenvironment{acknowledgements}{\par\bigskip\noindent\textbf{Acknowledgement}
\hskip 1mm}{\smallskip}
\newenvironment{question}{\bigskip\noindent\textbf{Question}
\hskip 1mm}{\smallskip}
\newenvironment{remark}{\bigskip\noindent\textbf{Remark}
\hskip 1mm}{\smallskip}

\DeclareMathOperator{\Aut}{Aut}

\DeclareMathOperator{\End}{End}
\DeclareMathOperator{\soc}{soc}

\DeclareMathOperator{\core}{core}

\title{A class of finite simple Bol loops of exponent 2} 
\thanks{This paper was written during the author's Marie Curie Fellowship
MEIF-CT-2006-041105.}  

\author{G\'abor P. Nagy} 
\email{nagyg@math.u-szeged.hu}
\address{Bolyai Institute, University of Szeged, Aradi v\'ertan\'uk tere 1,
  H-6720 Sze\-ged (Hungary)}
\curraddr{Mathematisches Institut, Universit\"at W\"urzburg, Am Hubland, D-97070
  W\"urzburg (Germany)}

\begin{document}

\begin{abstract}
In this paper we give an infinite class of finite simple right Bol loops of
exponent 2. The right multiplication group of these loops is an extension of an
elementary Abelian $2$-group by $S_5$. The construction uses the description of
the structure of such loops given by M. Aschbacher \cite{Aschbacher1}. These
results answer some questions of M. Aschbacher.
\end{abstract}

\maketitle

\section{Introduction} \label{sec:Intro}

The set $Q$ endowed with a binary operation $x\cdot y=xy$ is a
\emph{loop} if it has a unit element $1$ and the equation $xy=z$ has a unique
solution whenever two of the three unknowns are given. The loop $Q$ is a
\emph{(right) Bol loop} if the identity 
\[((xy)z)y=x(y(zy))\]
holds for all $x,y,z \in Q$. Bol loops of exponent $2$ play an important role
in the theory of loops and are related to interesting group theoretical
problems. Using the so called \emph{Baer correspondence,} this class of loops
can be described by a triple $(G,H,K)$ with group $G$, core-free subgroup $H\leq
G$ and system of coset representatives $K\subseteq G$ such that $1\in K$,
$G=\langle K \rangle$ and $K\setminus \{1\}$ is a union of conjugacy classes of
involutions.

Bol loops of exponent $2$ which are not elementary Abelian groups have long been
known to exist, the first construction is due to R. P. Burn \cite{Burn}. Later,
many infinite classes of such loops were given, see \cite{Kiechle, KiechleNagyG,
KolbKreuzer, NagyGfrattini}. All of these examples were solvable loops;
equivalently the group $G$ was a $2$-group. The existence of nonsolvable finite
Bol loops of exponent $2$ was considered as one of the main open problem in the
theory of loops and quasigroups. As the smallest such loop must be simple, this
question was related to the existence of finite simple proper right Bol loops. 
Here by proper we mean right Bol loops which are not Moufang, that is, which do
not satisfy the identity $x(yx)=(xy)x$. Finite simple proper Bol loops were
constructed by the author \cite{NagyGsimplebol} recently.

By \cite{NagyGsol}, the solvability of a Bol loop of $2$-power exponent is
equivalent to having $2$-power order. Later, S. Heiss \cite{Heiss} showed that
the solvability of the loop corresponding to the triple $(G,H,K)$ is equivalent
with the solvability of the group $G$. The next major step was the paper
\cite{Aschbacher1} by M. Aschbacher. His main result gives a detailed
description on the structure of the right multiplication group
of minimal nonsolvable Bol loops of exponent $2$. This result was achieved by
using the classification of finite simple groups. 

In this paper we apply Aschbacher's recipe to construct a class of finite simple
Bol loops of exponent $2$. In this way, we give a negative answer to questions 2
and 3 of \cite{Aschbacher1} and \cite{AschbKinyPhil}. The smallest member of our
class has order $96$. We emphasize that this example of is small and the
structural description of the smallest example in \cite{Aschbacher1} and
\cite{AschbKinyPhil} is so precise that it was only a matter of time that
somebody finds it by some short computer calculation. This explains the fact
that this loop was indepently discovered by the author and by B. Baumeister and
A. Stein \cite{BaumeisterStein} with a time delay of $10$ days. 

\begin{acknowledgements}
I would like to thank Michael Aschbacher (California Institute of Technology)
for his detailed comments on the previous version of this paper. He helped to
fill some gaps, to simplify the proofs and to replace computer based
calculations by theoretical arguments. I also would like to thank Peter M\"uller
(W\"urzburg) for many stimulating conversations and helpful comments.
\end{acknowledgements}

\section{Basic concepts}

In order to make this paper self-contained, we define the basic loop theoretical
concepts. Subloops, homomorphisms and normal subloops can be defined for the
class of loops similarly to groups, see \cite{Bruck}. The loop $Q$ is said to be
\emph{simple} if it has no proper normal subloops. 

For a loop $Q$, the maps $L_a,R_a:Q\to Q$, $L_a(x)=ax$, $R_a(x)=xa$ are the
\emph{left} and \emph{right multiplication maps} of $Q$. They are bijections of
$Q$ and generate the \emph{left} and \emph{right multiplication groups} of $Q$,
respectively. The \emph{section} $S(Q)$ of the loop $Q$ is the set of the right
multiplication maps of $Q$. 

The right Bol identity can equivalently be expressed by $R_xR_yR_x \in S(Q)$ for
all $x,y\in Q$. Bol loops are \emph{power-associative}, that is, $x^n$ is well
defined for all $x\in Q$ and $n \in \mathbb Z$. The \emph{order} of the element
$x$ is the smallest positive integer $n$ such that $x^n=1$. The \emph{exponent}
of $Q$ is the smallest positive integer $n$ such that $x^n=1$ for each $x\in
Q$. 

We will formulate our results by using the \emph{Baer correspondence} between
the class of loops and the class of triples $(G,H,K)$ where $G$ is a group,
$H\leq G$ and $K$ is a set of coset representatives of all conjugates $H^g$ in
$G$, and $1 \in K$. For a given loop $Q$, $G$ can be chosen to be the right
multiplication group, $H$ the stabilizer of the unit element and $K$ the set of
right multiplication maps. Conversely, for a given triple $(G,H,K)$, a loop can
be constructed in the following way. We define the operation $x\circ y$ on $K$
by $H(x\circ y)=Hxy$; then $Q=(K,\circ)$ is a loop. The triple satisfying the
above condition will be said to be a \emph{loop folder}. 

The loop folder $(G,H,K)$ determines a Bol loop of exponent $2$ if and only if
$K=\{1\}\cup \bigcup_{i \in I}C_i$, where the $C_i$'s are conjugacy classes of
involutions in $G$.

\section{The first construction} \label{sec:First}

As usual, $S_5$ and $PGL(2,5)$ are the permutation groups acting on $5$ and $6$
points, respectively. It is well known that $S_5 \cong PGL(2,5) \cong
\Aut(L_2(4)) \cong O_4^-(2)$, where $\Aut(L_2(4))$ is the extension of
$L_2(4)=PSL(2,4)\cong A_5$ by a field automorphism of order $2$, and $O_4^-(2)$
is the orthogonal group on a $4$-dimensional orthogonal space over $\mathbb F_2$
of Witt index $1$. We denote by $F_{20}$ the affine linear group acting on
$\mathbb{F}_5$, $F_{20} \cong C_5 \rtimes C_4$. On the one hand, $F_{20}$ is the
Borel subgroup of $PGL(2,5)$, that is, the stabilizer of a projective point. On
the other hand, $F_{20} \leq S_5$ is a sharply $2$-transitive permutation group
on $5$ points. 

In the sequel, we define a group $G$ which is a nonsplit extension of the
elementary Abelian group of order $32$ by $S_5$ such that the transpositions of
$S_5$ lift to involutions in $G$ and the even involutions of $S_5$ lift to
elements of order $4$. Despite the relatively small order of $G$, we found no
simple description for this group; therefore our definion will be 
rather \textit{ad hoc,} as well. We start with two technical lemmas. 

\begin{lemma} \label{lm:s5pres}
We have the following presentations of groups with generators and relations. 
\begin{eqnarray*}
A_5 &=&\langle a,b \mid a^2=b^3=(ab)^5=1 \rangle,\\
S_5 &=&\langle c,d \mid c^2=d^4=(cd)^5=[c,d]^3=1 \rangle,\\
2.S_5 &=&\langle C,D \mid C^2=D^8=(CD)^5=[C,D]^3=[C,D^4]=1 \rangle,
\end{eqnarray*}
where $2.S_5$ denotes the nonsplit central extension of $S_5$ in which the
transpositions lift to involutions. In other words, $2.S_5$ is the semidirect
product of $SL(2,5)=2.A_5$ with a group of order $2$. 
\end{lemma}
\begin{proof} The presentation for $A_5$ is well known. Assume $G=\langle
c,d\rangle$ is the group presented by the second set of relations above. We
observe first that $S_5$ satisfies these relations, hence no relation can
collapse. Put $a=d^2$ and $b=[d,c]=[c,d]^{-1}$. Then $a^2=b^3=1$ and
$ab=(dc)^2=((cd)^2)^c$. This latter implies $(ab)^5=1$, hence $G_0=\langle a,b
\rangle \cong A_5$. Moreover, since $dc=(dc)^6=(ab)^3 \in G_0$, we have
$cd=dc[c,d]=dcb^{-1} \in G_0$ and
\[dbd^{-1}=cdcd^{-1}=(cd)^2 d^{-2} \in G_0.\]
This means that $d$ and $cd$ normalize $G_0$, so $G_0\lhd G$. So $|G:G_0|=2$ and
$d\not\in G_0$ since $A_5$ contains no element of order $4$. This proves $G\cong
S_5$. Finally, $D^4$ is a central involution in $H=\langle C,D \rangle$ and
$H/\langle D^4\rangle$ maps surjectively to $S_5$. The order of $D$ is $8$,
thus, the extension is nonsplit and the involution $C$ covers a transposition
in $S_5$. 
\end{proof}

\begin{lemma} \label{lm:25S5_constr}
The permutations 
\[\begin{array}{rcl}
c&=&(1,4)(2,9)(3,10)(6,11)(7,12)(13,21)(14,22)(15,24)(16,23)(17,30)\\
&&(18,29)(19,31)(20,32)(33,35)(38,40), \\
d&=&(1,2,4,6,8,7,5,3)(9,13,25,18,10,14,26,17)(11,15,27,20,12,16,\\
&& 28,19)(21,30,38 , 34,23,31,40,35)(22,32,39,36,24,29,37,33)
  \end{array}\]
acting on $40$ points satisfy the relations
\begin{equation} \label{eq:relators}
c^2=d^8=(cd)^5=[c,d]^3=[d^4,c]^2=[d^4,cdcd^{-2}c]=1.
\end{equation}
Moreover, with $u_1=d^4$, $u_2=u_1^c$, $u_3=u_1^{cd}$, $u_4=u_1^{cdc}$,
$u_5=u_1^{cdcd}$, $u_6=u_1^{cdcdc}$ the identity $u_1u_2u_3u_4u_5u_6=1$ holds. 
\end{lemma}
\begin{proof} We leave the straightforward calculations to the reader. 
\end{proof}

\begin{lemma} \label{lm:25S5_struct}
The group $G= \langle c,d \rangle$ given in Lemma \ref{lm:25S5_constr}
satisfies	
\begin{description}
\item[(*)] $G$ has an elementary Abelian normal subgroup $J$ of order $32$ such that
$G/J\cong PGL(2,5)$ and $J$ is the $\mathbb F_2$-permutation module modulo its
center. Moreover, $[G,G]/[G,J]\cong SL(2,5)$ and $G$ splits over $[G,G]J$.
\end{description}
\end{lemma}
\begin{proof} We claim that the conjugacy class of $u_1=d^4$ in $G$ is $X =
\{u_1,\ldots,u_6\}$. It is immediate that $c$ induces the permutation
$(u_1u_2)(u_3u_4)(u_5u_6)$ on $X$. Moreover, $d$ centralizes $u_1$ and maps
$u_2 \mapsto u_3$, $u_4 \mapsto u_5$. From the last relation in
\eqref{eq:relators} follows $u_1^{cdc}=u_1^{cd^2}$, hence $u_3^d=u_3^c=u_4$.
By $[d^4,c]^2=1$ we have 
\[ u_1^{cd^4c}=cd^4cd^4cd^4c=d^4[d^4,c]^2= d^4=u_1, \]
thus $u_2=u_1^c=u_1^{cd^4}=u_4^{d^2}=u_5^d$. To see that $d$ acts on $X$, we
need to show that $d$ centralizes $u_6$:
\[u_6^d=u_1^{cdcdcd}=u_1^{(cd)^{-2}}=u_1^{d^{-1}cd^{-1}c}=
u_1^{cd^{-1}c}=u_2^{d^{-1}c}=u_5^c=u_6. \] 
The action of $d$ on $X$ is therefore $(u_2u_3u_4u_5)$. This not only shows that
$X$ is a conjugacy class in $G$, but we also have the action of $G$ on
$X$. Indeed, one shows by straightforward calculation that $\tilde
c=(12)(34)(56)$ and $\tilde d=(2345)$ satisfy the relations of $S_5$ from Lemma
\ref{lm:s5pres}. Since the action of $S_5$ on $6$ points is unique, we have
$G/C_G(X)\cong PGL(2,5)$.

As $[u_1,u_2]=[d^4,cd^4c]=[d^4,c]^2=1$ and $PGL(2,5)$ acts $2$-transitively,
$[u_i,u_j]=1$ holds for all $i,j$. This means that $J=\langle X \rangle$ is an
elementary Abelian $2$-group and $|J|=32$ by $u_1\cdots u_6=1$. Using the
presentation of $2.S_5$ from Lemma \ref{lm:s5pres}, $G/J_0\cong 2.S_5$. This
implies 
\[[G,G]/[G,J]\cong [G/J_0,G/J_0] \cong 2.A_5\cong SL(2,5).\] 
Finally, $G$ splits over $[G,G]J$ as $c \not\in [G,G]J$. 
\end{proof}

In the sequel, $G$ will denote a group satisfying (*). We would like to make
clear that it can be shown using the computer algebra system GAP \cite{GAP} that
the group given in Lemma \ref{lm:25S5_constr} is the unique group with this
property. However, we hope that this more general approach will help in future
generalization of the constructions of this paper. 

Among other properties, we show in the next lemma that for our group $G$,
$G'=[G,G]$ is a perfect group. Actually, we found $G'$ by using the library of
perfect groups in the computer algebra system GAP \cite{GAP} and constructed $G$
as a split extension of $G'$ by an outer automorphism of order $2$.

\begin{lemma} \label{lm:Gprops}
Let $G$ be a group satisfying (*) and define $J_0=[G,J]$. 
\begin{enumerate}[(i)]
\item We have $G''=G'=[G,G]=[G,G]J$ and $|G:G'|=2$. 
\item $G\setminus J$ contains a unique class $c^G$ of involutions,
and $|c^G|=80$. In particular, all involutions of $G'=[G,G]=[G,G]J$ lie in $J$. 
\item Let $P$ be a Sylow $5$-subgroup. Then $N_{J_0}(P)=\{1\}$ and
$N_G(P)\cong C_8 \ltimes C_5$. Moreover, if the subgroup $U\leq G$ maps onto
$F_{20}$ modulo $J$ then $U=N_G(P)$ or $U=N_G(P)J_0$. 
\end{enumerate}
\end{lemma}
\begin{proof} 
(i) Let $V$ be the permutation $\mathbb F_2$-module of $PGL(2,5)$ with basis
$\{u_1,\ldots,u_6\}$. Due to the $2$-transitivity, the orbit of the element
$u_1+u_2$ consists of the elements $u_i+u_j$, $i\neq j$ which are different
modulo the center $\langle u_1+\ldots+u_6 \rangle$ of $V$. Hence both $PGL(2,5)$
and $PSL(2,5)$ act transitively on the nonidentity elements of $J_0=[G,J]$,
which implies that $J_0$ is a minimal normal subgroup in $G$ and $[G,G]J$. It
follows that $J_0=[G',J_0]$ and $G''/J_0=(G'/J_0)'=G'/J_0$ by
$SL(2,5)'=SL(2,5)$. This means $G''=G'$. Finally, $J\leq G'$ follows from
$J/J_0=Z(G/J_0)\leq (G/J_0)'=G'/J_0$. 

(ii) Since $G$ splits over $G'=G'J$ we can take an involution $c$ from
$G\setminus G'$; the image of $c$ in $G/J\cong S_5$ is a transposition. As $J$
is the permutation module modulo its center, $\dim_{\mathbb F_2}(C_{J_0}(c))=2$
and $\dim_{\mathbb F_2}(C_{J}(c))=3$. It is easy to check that $2.S_5\cong
G/J_0$ contains $20$ non-central involutions and they are all conjugate. 

Let $c'$ be another involution in $G\setminus J$; we want show that $c,c'$ are
conjugate. For some $g\in G$, $(cJ_0)^{gJ_0}=c'J_0$, that is, $c^g\in
c'J_0$. Hence we can assume $c\in c'J_0$, $c'=cj$ with $j \in J_0$. The element
$cj$ has order $2$ if and only if $j\in C_{J_0}(c)$. On the one hand, $c^{J_0}
\subseteq cC_{J_0}(c)$. On the other hand,
\[|c^{J_0}|=|J_0:C_{J_0}(c)|=4=|C_{J_0}(c)|=|cC_{J_0}(c)|.\]
This implies $c^{J_0}=cC_{J_0}(c)$ and $c' \in c^{J_0}$. No involution of $G'$
can be conjugate to $c$, hence all involutions of $G'$ must lie in $J$. Finally,
we show $|c^G|=80$. As $c^g\in cJ_0$ if and only if $c^g= c^j$
for some $j\in J_0$, we have $N_G(cJ_0)=C_G(c)J_0$. Moreover,
$C_{G/J_0}(cJ_0)=N_G(cJ_0)/J_0$. Thus,
\begin{eqnarray*}
|G:C_G(c)| &=& |G:N_G(cJ_0)||C_G(c)J_0:C_G(c)| \\
 &=& |G/J_0:C_{G/J_0}(cJ_0)||J_0:C_{J_0}(c)| \\
 &=& |(cJ_0)^{G/J_0}||c^{J_0}|=20\cdot 4=80. 
\end{eqnarray*}

(iii) $P$ acts fixed point free on the involutions of $J_0$, thus, $N_{J_0}(P) =
1$. Moreover, $5\nmid |J|-1$, hence $P$ centralizes a unique element $a \in J$.
Let $U$ be a preimage of $F_{20}$ modulo $J$ and put $\bar U=U/\langle a
\rangle$, $\bar J=J/\langle a \rangle$. Then $\bar J$ is a minimal normal
subgroup of $\bar U$. Since $F_{20}=\bar U/\bar J$ acts faithfully on $\bar J$,
we have $C_{\bar U}(\bar J)=\bar J$. By \cite[II.3.3. Satz]{Huppert}, $\bar J$
has a complement $\bar H$ in $\bar U$, $\bar H \cong F_{20}$. Let $H$ be the
preimage of $\bar H$, then $H$ has a unique (hence normal) $5$-Sylow and $H\cong
C_8 \ltimes C_5$. This shows $N_G(P)\cong C_8 \ltimes C_5$. For the last
statement, record that $U\cap J_0$ is either $1$ or $J_0$. 
\end{proof}

The following proposition will apply in all of our examples of Bol loop folders
of exponent $2$. We hope that it will also apply in future constructions not
considered here. Recall that $O_2(G)$ is the largest normal $2$-subgroup of $G$.

\begin{proposition} \label{pr:general}
Assume $G$ is a finite group, $J=O_2(G)$ and $G^+=G/J\cong S_5$. We denote by
$g^+$ the element of $S_5$ corresponding to $gJ$. Set $L=G'J$, $K_1$ the
involutions in $G\setminus L$, $K_0$ a $G$-invariant subset of $J$ containing
$1$ such that $K_0\setminus \{1\}$ consists of involutions, and $H\leq G$. Set
$K=K_0\cup K_1$, $n_0=|K_0|$, and $n_1=|K_1\cap aJ|$ for $a \in K_1$. Assume
\begin{enumerate}[(a)]
\item $(J,H\cap J,K_0)$ is a Bol loop folder of exponent $2$. 
\item $n_0=2n_1$.
\item $|G^+:H^+|=6$.
\item For each $a\in K_1$, $C_{H\cap J}(a)=1$. 
\item Every involution of $L$ is contained in $J$. 
\end{enumerate}
Then $(G,H,K)$ is a Bol loop folder of exponent $2$, and $|K|=6n_0=12n_1$. 
\end{proposition}
\begin{proof} First $K_1^+$ is the set of transpositions of $S_5$, so $|K_1^+|=10$.
This implies that $n_1$ is well defined. Indeed, for $a,b\in K_1$, $aJ, bJ$ are
conjugate, hence $K_1\cap aJ, K_1\cap bJ$ are conjugate in $G$. Moreover,
$|K_1|=10n_1$ and by (b), 
\[|K|=|K_0|+|K_1|=2n_1+10n_1=12n_1=6n_0.\]
Next by (a) and (c), 
\[|G:H|=|G:HJ||HJ:H|=|G^+:H^+||J:J\cap H|=6|K_0|=6n_0,\]
so $|G:H|=|K|$. 

We claim $xy\not\in H$ for distinct $x,y\in K$. If so, as $|G:H|=|K|$, $K$ is a
set of coset representatives for $H$ in $G$. Then as $K$ is $G$-invariant and
$K\setminus \{1\}$ consists of involutions, $(G,H,K)$ is a Bol loop folder of
exponent $2$. 

If $x,y\in J$ then $x,y\in K_0$, so $xy\not\in H$ by (a). Next $K_1^+ \cap
H^+=\emptyset$, so if $x\in J$ and $y\in K_1$ then $(xy)^+=y^+ \not \in H^+$, so
$xy \not \in H$. Thus we may take $x,y\in K_1$ and $xy\in H$. Now as $K_1^+$ is
a set of transpositions in $S_5$, the order of $(xy)^+$ is $1,2$ or $3$. Since
$H^+\cong F_{20}$ has no element of order $3$, we get $(xy)^2 \in J$. In
particular, $D=\langle x,y \rangle$ is a $2$-group. Let $z$ be the unique
involution in $\langle xy \rangle$. By $xy \in L$ and (e), $z\in H\cap J$.
Moreover, $x,y$ commute with $z$, which contradicts to (d). 
\end{proof}

\begin{remark}
The fact that (e) is necessary can be seen from the counterexample $G=S_5\ltimes
J$. 
\end{remark}

\begin{theorem} \label{th:firstexample}
Assume $G$ is a group satisfying condition (*) of Lemma \ref{lm:25S5_struct}.
Let $J_0$ be the minimal normal subgroup of $G$ and put $K=J_0 \cup c^G$.
Define $H=N_G(P)$ where $P$ is a 5-Sylow subgroup of $G$. Then $(G,H,K)$ is a
Bol loop folder determining a simple Bol loop of exponent $2$ of order $96$.
Conversely, if $(G,H^*,K^*)$ is an exponent $2$ Bol loop folder then $H^*$ is
a conjugate of $H$ and $K^*=K$. 
\end{theorem}
\begin{proof} With the notation of Proposition \ref{pr:general}, $K_0=J_0$ and
$K_1=c^G$. Then $n_0=16$, $n_1=|c^G \cap cJ|=80/10=8$ and $|G^+:H^+|=6$, so
(b) and (c) hold. (e) follows from Lemma \ref{lm:Gprops}(ii). Since $J$ is
elementary Abelian, $H\cap J$ consists of $1$ and the unique involution of $H$.
This involution cannot be centralized by $c$, otherwise it would be central in
$G=\langle c,P,J\rangle$; hence (d). Finally, $H\cap J$ is not contained in
$J_0$, therefore $J_0$ is a complement to $H\cap J$ in $J$; showing (a). By
Proposition \ref{pr:general}, $(G,H,K)$ is a Bol loop folder of exponent $2$. 

For the converse, we observe that $(G,H^*,K^*)$ determines a Bol loop of
exponent $2$ with all proper subloops solvable. Thus, by the Main Theorem of
\cite{Aschbacher1}, $H^*$ maps surjectively to $F_{20}$. By Lemma
\ref{lm:Gprops}(iii), $H^*=H$ or $H^*=HJ_0$ up to conjugaction. In the latter
case, the loop has order $6$ which is impossible. Again by Aschbacher's result,
$c^G \subset K^*$. Finally, if $J_0\nsubseteq K^*$, then $K^*$ will contain a
conjugate of the involution of $H$, which is not possible. This proves the
theorem. 
\end{proof}

As the group given in Lemma \ref{lm:25S5_constr} satisfies (*), we have: 

\begin{corollary}
There exists a simple Bol loop of exponent $2$ and order $96$. \qed
\end{corollary}

\begin{remark} 
The Bol loop folder $(G,H,K)$ of Theorem \ref{th:firstexample} was discovered
independently by B. Baumeister and A. Stein \cite{BaumeisterStein} (Free
University of Berlin), as well. 
\end{remark}

\section{$S_5$-modules over $\mathbb F_2$} \label{sec:S5modules}

In this section we collect some useful facts about $kS_5$-modules, where $k$ is
a field of characteristic $2$.

\begin{lemma} \label{lm:S5modules}
The group $S_5$ has three absolutely irreducible representations over $\mathbb
F_2$: the trivial representation and two representations $M,N$ of dimension
4. The two $4$-dimensional modules can be distinguished by the fact that
$C_M(x)=0$ and $\dim_{\mathbb F_2}(C_N(x))=2$ for an element $x\in S_5$ of order
$3$. Moreover, the following hold.
\begin{enumerate}[(i)]
\item $M$ is the $4$-dimensional irreducible component in the $6$-dimensional
  permutation module for $S_5\cong PGL(2,5)$. Also, let $V$ be the natural
  $2$-dimensional module of $A_5\cong SL(2,4)$ over the field $\mathbb F_4$ and
  $\sigma$ be semilinear map of $V$ induced by the Frobenius automorphism of
  $\mathbb F_4$. Then $S_5\cong SL(2,4)\rtimes \langle \sigma \rangle$ and $V$
  is a $4$-dimensional $S_5$-module over $\mathbb F_2$. The $S_5$-modules $M$
  and $V$ are isomorphic.
\item $N$ is is the $4$-dimensional irreducible component in the $5$-dimensional
  permutation module of $S_5$. Also if $N$ is a $4$-dimensional orthogonal space
  of Witt index $1$ over $\mathbb F_2$, then $O(N)=O^-_4(2)\cong S_5$. Note that
  $N$ has $5$ singular and $10$ nonsingular vectors and these are the
  $S_5$-orbits on $N$. 
\item $N$ is absolutely irreducible as $A_5$-module. $M$ is irreducible but not
  absolutely as an $A_5$-module, the splitting field being $\mathbb F_4$. In
  particular, the modules are nonisomorphic as $A_5$-modules.
\item $N$ and $M$ are isomorphic absolutely irreducible projective $F_{20}$-modules. 
\end{enumerate}
\end{lemma}
\begin{proof} Let us first define $N,M$ as irreducible components of the permutation
modules. By \cite[Table 1]{Mortimer}, they are absolutely irreducible. As $S_5$
has $3$ classes of elements of odd order, by \cite[Theorem 3.2]{Alperin} $S_5$
has no other absolutely irreducible modules over $F_2$. The properties of $N,M$
can be verified by straightforward calculations, the irreducibility as $A_5$ and
$F_{20}$-modules follows again from \cite[Table 1]{Mortimer}. We show
$N_{F_{20}}\cong M_{F_{20}}$. As $F_{20}$ has two classes of elements of odd
order, $F_{20}$ has two absolutely irreducible modules: the trivial one and
$N_{F_{20}}$ coming from the $2$-transitive permutation representation. So if
$M_{F_{20}}$ were not isomorphic to $N_{F_{20}}$ then it could be brought to
upper triangular form over $\bar{\mathbb F}_2$, which is clearly impossible.

It remains to show that $N_{F_{20}}$ is projective. Since $N_{C_4}$ is
isomorphic to the group algebra $\mathbb F_2C_4$, it is a projective
$C_4$-module by \cite[Theorem 4.2]{Alperin}. Using \cite[Corollary
9.3]{Alperin}, we obtain that $N$ is projective as $F_{20}$-module. 
\end{proof}

We observe that these $S_5$-modules can immediately be constructed using the
Steinberg Tensor Product Theorem \cite[Theorem 13.1]{Steinberg}, as well. 

\medskip

We will now construct an $S_5$-module $U$ which will play a central role in the
generalization of our first construction of a Bol loop of exponent $2$.

Let $U=U_1\oplus U_2$ be the direct sum of two copies of $N$ as an $\mathbb
F_2A_5$-module. As $U_i$ is an orthogonal space, we can regard $U$ as an
orthogonal space which is the orthogonal direct sum of the two nondegenerate
subspaces $U_1$ and $U_2$. The stabilizer $B$ of $\{U_1,U_2\}$ in $O(U)$ is
$(G_1 \times G_2)\langle \tau \rangle$ where 
\[G_i=C_{O(U)}(U_{3-i})\cong O(U_i) \cong S_5\]
and $\tau$ is an involution interchanging $U_1,U_2$. Thus $B$ is the wreath
product of $S_5$ with $C_2$. In particular, the elements $\tau$ and 
$(g_1,g_2) \in G_1G_2$ map $u_1\oplus u_2 \in U$ to
\[(u_1\oplus u_2)\tau = u_2\oplus u_1, \mbox{ and } (u_1\oplus u_2)(g_1,g_2) =
u_1g_1\oplus u_2g_2,\]
respectively. Set 
\[G_0=C_{G_1G_2}(\tau)=\{(g,g) \mid g \in S_2\} \cong S_5,\]
and let $L=[G_0,G_0]\cong A_5$, $t_0$ an involution (transposition) in
$G_0\setminus L$, $t=t_0\tau$, and $D=L\langle t \rangle$. Then $t\tau=\tau t$
and the action of $t$ on $U$ is 

\[(u_1\oplus u_2)t=u_2c \oplus u_1c,\]
where $c\in S_5$ is the transposition corresponding to $t_0$. It is immediate
that $D\cong S_5$. 

Set $W=C_U(\tau)$. Then $W$ is an $\mathbb F_2D$-submodule of $U$, and also 
\[W= [U,\tau] = \{u+u\tau \mid u\in U_1\} = \{u\oplus u \mid u \in N\},\]
with the map $u\mapsto [u,\tau]=u+u\tau$ an $\mathbb F_2L$-isomorphism of
$U_1$ with $W$. If $Q$ is the quadratic form on $U$ then as $Q(u_1\oplus
u_2)=Q(u_1)+Q(u_2)$ and $Q(u)=Q(u\tau)$, $[u,\tau]$ is singular, so $W$ is
totally singular. 

\begin{lemma} \label{lm:Usubmodules}
With the notation above, we have:
\begin{enumerate}[(i)]
\item $U$ has $3$ irreducible $L$-submodules, namely $U_1,U_2$ and $W$. 
\item $W$ is the unique proper $D$-submodule of $U$.
\item Let $P$ be a Sylow $5$-subgroup of $D$, $D_1=N_D(P)$. Then $D_1\cong
F_{20}$ and $U$ has precisely $3$ $D_1$-submodules $W,T_1,T_2$. 
\item The orbits of $D_1$ on $T_i$ have length $1,5,10$. In particular, each
member of $T_i$ is fixed by some involution of $D_1$. 
\end{enumerate}
\end{lemma}
\begin{proof} By Lemma \ref{lm:S5modules}, $N$ is projective as $F_{20}$-module.
Since $U/W$ is a $4$-dimensional irreducible for $D_1$, $U_{D_1}$ splits over
$W_{D_1}\cong N_{D_1}$ and hence $U=W\oplus T_1$ with $D_1$-submodule $T_1$.
Again by Lemma \ref{lm:S5modules}, $N_L$ and $N_{D_1}$ are completely
irreducible, Schur's lemma then implies $\End_{\mathbb F_2L}(N)=\End_{\mathbb
F_2D_1}(N)=\mathbb F_2$. We can now apply \cite[(27.14)]{FGT} to obtain (i) and
(iii). (ii) follows from (i). Finally, (iv) holds since $T_i$ and $W$ are
$D_1$-isomorphic and $W$ is the permutation module modulo the center. 
\end{proof}

In the next lemma, we keep using the above notation.

\begin{lemma} \label{lm:technical}
\begin{enumerate}[(i)]
\item $\dim C_U(t)=4$, $C_U(t)+T_1=U$ and $C_U(t) \cap T_1=0$.
\item Under the action of $A_5$ on the submodules $W,U_1,U_2$ of $U$, the
lengths of the orbits are $1,5,10$. Let $S_0,S_1,S_2$ be the orbits of
length $5$ in $W,U_1,U_2$, respectively. Then $S=\{0\}\cup S_0 \cup S_1 \cup
S_2$ is a (nonlinear) $S_5$-invariant complement to $T_1$ in $U$, that is,
$S+T_1=U$.
\end{enumerate}
\end{lemma}
\begin{proof} (i) We have $D=L\langle t \rangle \cong S_5$. Let us denote the element
of $S_5$ corresponding to $a\in D$ by $a^+$, w.l.o.g. we can assume $t^+=(12)$
and $P^+=\langle (12345) \rangle$. Then $D_1^+=\langle (12345), (1325)
\rangle$. Let $b\in D$ such that $b^+=(12)(35)$. Then $b \in D_1$ and $b$
commutes with $t$ and $\tau$. Record that the action of $b,t,\tau$ on $U$ is
\begin{eqnarray*}
b&:&u_1\oplus u_2 \mapsto u_1b^+\oplus u_2b^+,\\
\tau&:&u_1\oplus u_2 \mapsto u_2\oplus u_1,\\
t&:&u_1\oplus u_2 \mapsto u_2t^+\oplus u_1t^+.
\end{eqnarray*}
Define the element $\tilde c \in B$ by 
\[\tilde c:u_1\oplus u_2 \mapsto u_1t^+\oplus u_2.\]
Then $\tilde c$ commutes with $b$ and $\tau^{\tilde c}=t$. Put $E_1=\langle t,b
\rangle$, clearly $E_1^{\tilde c}=\langle \tau,b\rangle$. On the one hand, we
have
\begin{eqnarray*}
\dim_{\mathbb F_2}(C_U(E_1)) &=& \dim (C_U(\tau,b)) \\
&=& \dim C_{C_U(b)}(\tau) \\
&=& \dim C_{C_{U_1}(b) \oplus C_{U_2}(b)}(\tau) \\
&=& \dim C_{U_1}(b) \\
&=& \dim C_N(b^+)=2.
\end{eqnarray*}
We show on the other hand that $\dim_{\mathbb F_2}(C_W(E_1))\geq 2$. Indeed, the
$S_5$-modules $W$ and $N$ are isomorphic and $\dim(C_W(t))=\dim(C_N(t^+))=3$.
Then $C_W(E_1)=C_{C_W(t)}(b)$ is of rank at least $\dim(C_W(t))/2=3/2$. 

Now, from $\dim(C_W(E_1))\geq 2$ follows $C_U(E_1)\leq W$. However if
$C_{T_1}(t)\neq 0$ then $C_{T_1}(E_1)\neq 0$, contradicting $T_1\cap W=0$ and
$C_U(E_1)\leq W$. 

(ii) We have seen that $S_i$ is the set of singular points in $U_i$ for
$i=1,2$. The action of $D_1\cong F_{20}$ on $S_i$ is its natural $2$-transitive
action on $5$ points. $D_1\cap L$ contains precisely $5$ involutions and each
member of $S_i$ is fixed by exactly one involution of $D_1 \cap L$. Moreover,
each member of $T_1$ is fixed by some involution of $D_1$. 

We have to show that for distinct $x,y\in S$, $x+y\not\in T_1$; then $S$ is a
complement to $T_1$ by an order argument. Assume $x+y\in T_1$ and denote by $a$
an involution of $D_1$ fixing $x+y$. Then $x,y$ are the projections of $x+y$ on
$S_i,S_j$, so $a$ fixes the projections $x$ and $y$. As $T_1\cap U_i=0$ for
$0\leq i\leq 2$, $x\in S_i$ and $y \in S_j$ for some $i\neq j$. If $x\in S_1$
and $y\in S_2$, then $x$ and $y$ are the unique fixed points of $a$ in
$S_1,S_2$. By $a\tau=\tau a$, $x\tau = y$ holds, and hence $x+y=[x,\tau]\in W$,
contradicting $T_1\cap W=0$. 

Thus we may take $x\in S_0$ and $y\in S_1$. Then $x=x_1+x_2$ with $x_i=S_i$, and
as $x\in S_0$, $x_2=x_1\tau$ with $x_i\in S_i$. Now, $x_i$ is the unique fixed
point of $a$ in $S_i$ and $y$ the unique fixed point of $a$ in $S_1$, so $y=x_1$
and $x+y=x_1+x_2+y=x_2\in T_1\cap U_2=0$, a contradiction. 
\end{proof}

\section{An infinite family of simple Bol loops of exponent 2}

In this section, $G$ denotes a group satisfying condition (*) of Lemma
\ref{lm:25S5_struct}, $H$ is the normalizer of a $5$-Sylow $P$ of $G$, $c$ an
involution from $G\setminus [G,G]J$. The $S_5$-modules $N$ and $U$ are defined
as in Section \ref{sec:S5modules}. Also $U=U_1 \oplus U_2=T_1 \oplus T_2$ where
$U_1,U_2$ are $A_5$-submodules, and $T_1,T_2$ are $F_{20}$-submodules. Moreover,
$U_1,U_2,T_1,T_2$ are different from the unique $S_5$-submodule $W$ of $U$.
All these subspaces are irreducible $\mathbb F_2P$-modules, which implies
$U_i \cap T_j =0$.

Let us fix a positive integer $k$ and put 
\[\mathscr U=U^k,\; \mathscr U_i=U_i^k,\; \mathscr T_i=T_i^k,\; \mathscr W =
W^k.\]
Clearly, $\mathscr W$ is a $S_5$-submodule and $\mathscr U=\mathscr U_1 \oplus
\mathscr U_2=\mathscr T_1 \oplus \mathscr T_2$. We write $\mathscr{G}=G\ltimes
\mathscr U$ where $J\triangleleft G$ acts trivially on the $S_5$-module
$\mathscr U$. Moreover, 
\[\mathscr J=O_2(\mathscr G)=\langle J, \mathscr U \rangle.\] 
We will consider the elementary Abelian subgroup $\mathscr J$ of $\mathscr G$ as
an $S_5$-module over the field $\mathbb F_2$. In particular, with some abuse of
notation, we will denote the group operation on $\mathscr J$ additively and
write $\mathscr J = J+\mathscr U$, etc. It is easy to see that 
\[\mathscr G/\mathscr J\cong S_5 \mbox{ and } \soc(\mathscr G)= J_0+\mathscr
W.\]
Moreover, since $J_0$ and $W$ are non-isomorphic $S_5$-modules, $J_0\oplus W$
does not contain diagonal submodules. This implies that for any minimal
submodule $M$ of $\mathscr J$, we either have $M=J_0$ or $M\leq \mathscr W$. 

The action of the involution $c\in G\setminus J$ on $\mathscr U$ equals the
action of the transposition $(12) \in S_5$, hence $c$ interchanges $\mathscr
U_1, \mathscr U_2$. This implies $|C_{\mathscr U}(c)|=16^k$ and $|c^{\mathscr
U}|=16^k$; that is, $\mathscr U$ is transitive on the involutions in $c\mathscr
U$. As $G\cong \mathscr G/\mathscr U$ is transitive on the $80$ involutions in
$G\setminus J$, and as $\mathscr U$ is transitive on the $16^k$ involutions on
$c\mathscr U$, $\mathscr G$ is transitive on the $80\cdot 16^k$ involutions in
$\mathscr G \setminus \mathscr J$. 

Using Aschbacher's Main Theorem \cite{Aschbacher1} we conclude that a Bol loop
folder $(\mathscr G, \mathscr H, \mathscr K)$ must have the following
properties: The index of $\mathscr H$ has to be $96\cdot 16^k$, that is,
$\mathscr H$ must have order $40\cdot 16^k$. The set of involutions $\mathscr K$
is the union of $c^{\mathscr G}$ and $\mathscr K \cap \mathscr J$, thus,
$|\mathscr K \cap \mathscr J|=16\cdot 16^k$.

There are very many possible choices for $\mathscr H$ and $\mathscr K$.The most
obvious choice is the following.
\begin{proposition} \label{pr:splitcase}
Put $\mathscr H=H\ltimes \mathscr T_1$ and $\mathscr K=c^{\mathscr G} \cup
(J_0\oplus \mathscr W)$. Then $(\mathscr G, \mathscr H, \mathscr K)$ is a Bol
loop folder. Moreover, the homomorphism $\mathscr G \to G$ with kernel $\mathscr
U$ induces a surjective homomorphism between the loop folders $(\mathscr G,
\mathscr H, \mathscr K)$ and $(G,H,K)$. In other words, the Bol loop
corresponding to $(\mathscr G, \mathscr H, \mathscr K)$ is an extension of the
elementary Abelian group of order $2^{2k}$ by the loop corresponding to
$(G,H,K)$. 
\end{proposition}
\begin{proof} We apply Proposition \ref{pr:general}, (a), (c) and (e) are trivial. (b)
follows from 
\[n_1=|c^{\mathscr G}\cap c\mathscr J|=|C_{\mathscr J}(c)|=8|C_{\mathscr
  U}(c)|=8\cdot16 ^k\]
and $n_0=|J_0||\mathscr W|=16\cdot 16^k=2n_1$. For (d), we use
\[C_{\mathscr H\cap \mathscr J}(c)=C_{H\cap J+\mathscr T_1}(c)=C_{H\cap
  J}(c)=1,\] 
as by Lemma \ref{lm:technical}, $\mathscr T_1\cap \mathscr T_1^c=0$. 
\end{proof}

\medskip

In the rest of this section, for each integer $k\geq 1$, we modify these
$\mathscr H$ and $\mathscr K$ such that the resulting loop will be simple.  Let
$U^*$ be a copy of $U$ in $\mathscr U$ such that $\mathscr U=U^*\oplus
U^{k-1}$. We denote the subspaces corresponding to $T_i,U_i,W$ by $T_i^*, U_i^*,
W^*$. Let us define the set $S\subseteq U$ as in Lemma \ref{lm:technical}(ii)
and let $S^*$ be the corresponding subset of $U^*$. In order to construct the
new $\mathscr K$, we simply replace $W^*$ by $S^*$.

Let $\psi:J_0\to T_1^*$ be an isomorphism of $F_{20}$-modules and define
\[ \mathscr T_\psi = \{ v+\psi(v)+u \mid v \in J_0, u \in T_1^{k-1}\}. \]
Then $\mathscr T_\psi$ is normalized by $H$ and we define the new subgroup
$\mathscr H$ of $\mathscr G$ by $\mathscr H = H \ltimes \mathscr T_\psi$.

\begin{theorem}
Let $\mathscr H=H \ltimes \mathscr T_\psi$, $\widetilde{\mathscr W}=(\mathscr W
\setminus W^*) \cup S^*$ and $\mathscr K=c^{\mathscr G}\cup (J_0+
\widetilde{\mathscr W})$. Then the triple $(\mathscr G, \mathscr H, \mathscr K)$
is a Bol loop folder such that the corresponding Bol loop is simple of exponent
$2$.
\end{theorem}
\begin{proof} Again, the first statement follows from Proposition \ref{pr:general} and
Lemma \ref{lm:technical}; one needs to verify hypothesis (a) and (d) of the
Proposition only. Again (d) follows from $T_1\cap T_1^c=0$. 

We prove (a) by showing that $\mathscr K_0= \mathscr K \cap \mathscr H \mathscr
J$ is a transversal to $\mathscr H$ in $\mathscr H \mathscr J$. Since $\mathscr
K_0 \subset \mathscr J$, this is equivalent with the fact that $\mathscr
K_0=J_0+\widetilde{\mathscr W}$ is a complement of the subspace
\[\mathscr H \cap \mathscr J=H\cap J+\mathscr T_\psi\] in
$\mathscr J$, \textit{i.e.}
\[\mathscr K_0 + \mathscr H \cap \mathscr J = \mathscr J.\]
In order to show this, we use the identities
\begin{eqnarray*} 
\mathscr W+\mathscr T_1&=&\mathscr U, \\
J_0+\mathscr T_\psi &=& J_0+\mathscr T_1, \\
J_0+\mathscr H \cap \mathscr J &=& J_0+H\cap J+\mathscr T_\psi \\
&=& J_0+H\cap J+\mathscr T_1.
\end{eqnarray*}
Then
\begin{eqnarray*} 
J_0+\mathscr W+ \mathscr H \cap \mathscr J&=& J_0+\mathscr W + H\cap J+ \mathscr
T_\psi \\
&=& J_0+\mathscr W + H\cap J+ \mathscr T_1 \\
&=& J_0+\mathscr U+H\cap J\\
&=& J+\mathscr U\\
&=&\mathscr J,
\end{eqnarray*}
that is, $J_0+\mathscr W$ is a complement to $\mathscr H \cap \mathscr J$ in
$\mathscr J$ by the order argument $|J_0+\mathscr W||\mathscr H\cap \mathscr
J|=|\mathscr J|$. 

We constructed $\widetilde{\mathscr W}$ from $\mathscr W$ by
deleting the minimal submodule $W^*$ and replacing it by $S^*$. Therefore,
it is enough to show that this deformations of $\mathscr W$ do not change the
property of being a complement. 
\begin{eqnarray*}
J_0+W^*+\mathscr H \cap \mathscr J &=& J_0+W^*+H\cap J+\mathscr T_\psi \\
&=& H\cap J+J_0+W^*+\mathscr T_1 \\
&=& H\cap J+J_0+W^*+T_1^*+\mathscr T_1 \\
&=& H\cap J+J_0+S^*+T_1^*+\mathscr T_1 \\
&=& J_0+S^*+H\cap J+\mathscr T_\psi \\
&=& J_0+S^* +\mathscr H \cap \mathscr J
\end{eqnarray*}
This proves (a), hence $(\mathscr G, \mathscr H, \mathscr K)$ is a Bol loop folder.

It remains to show that the Bol loop $\mathscr Q$ corresponding to $(\mathscr G,
\mathscr H, \mathscr K)$ is simple. Let us therefore assume that $\mathscr Q\to
\mathscr Q^\sharp$ is a nontrivial surjective loop homomorphism and let
$(\mathscr G^\sharp, \mathscr H^\sharp, \mathscr K^\sharp)$ be the loop folder
of $\mathscr Q^\sharp$. Then we have a surjective homomorphism $\alpha:\mathscr
G \to \mathscr G^\sharp$ with $\alpha(\mathscr H)=\mathscr H^\sharp$ and
$\alpha(\mathscr K)=\mathscr K^\sharp$. Let $\mathscr N=\ker \alpha$ and
$c^\sharp=\alpha(c)=c\mathscr N$. On the one hand, $\mathscr H^\sharp$ is
core-free, thus,
\begin{equation} \label{eq:N=core}
\core_{\mathscr G}(\mathscr H \mathscr N)=\mathscr N.
\end{equation}
On the other hand, $\mathscr N \leq \mathscr J$ since otherwise $\mathscr H
\mathscr N=\mathscr G$ and $\mathscr Q^\sharp=1$. 

Let us first assume that $J_0\leq \mathscr N$. Since
\[J/J_0 \leq Z(\mathscr G/J_0) \lhd \mathscr G/J_0 \mbox{ and } J/J_0 \leq
\mathscr H \mathscr N/J_0,\] 
we have $J\leq \mathscr N$ by \eqref{eq:N=core}. In this case the image
$G^\sharp$ of $G\leq \mathscr G$ is a homomorphic image of $G/J\cong
S_5$. Furthermore, if $[c^\sharp, {c^\sharp}^g]=1$ then $c^\sharp {c^\sharp}^g$
normalizes a Sylow $5$-subgroup of $G^\sharp$, thus, $c^\sharp {c^\sharp}^g$
is contained in a conjugate of $\mathscr H^\sharp$, and hence
$c^\sharp={c^\sharp}^g$ in this case. As the commuting graph of transpositions
in $S_5$ is connected, $c^\sharp= {c^\sharp}^g$ for all $g$. This means
$[c,\mathscr G]\leq \mathscr N$, contradicting to $\mathscr N\leq \mathscr J$. 

Let us now assume $J_0 \nleq \mathscr M$ and let $M$ be a minimal normal
subgroup of $\mathscr G$ contained in $\mathscr N$. Then $M\leq \soc(\mathscr
G)=J_0 + \mathscr W$. Since $J_0$ and $W$ are non-isomorphic $S_5$-modules,
$J_0+\mathscr W$ contains no submodules isomorphic to $J_0$ and different from
$J_0$. This implies $M\leq \mathscr W$ and, in particular, $\mathscr N \cap
\mathscr W\neq 0$.

Let us take an element $s \in S^*\setminus W^* \subseteq \mathscr K$. As $W^*$
and $T_0^*$ are complements in $U^*$, $s\neq 0$ has the unique decomposition
$s=w+t$ with $0\neq w\in W^*$ and $0\neq t\in T_1^*$. Furthermore, for $0\neq
j=\psi^{-1}(t) \in J_0$, $j+t\in \mathscr T_\psi \leq \mathscr H$ holds. We
claim that $j+t\in \mathscr N$. Indeed, we have the decomposition
\[w=(s+j)+(j+t), \hskip 1cm s+j \in \mathscr K, j+t\in \mathscr H.\]
If $\mathscr N\leq U^*$ then $M=W^*$ and $\alpha(s+j)=\alpha(j+t)\in \mathscr
K^\sharp \cap \mathscr H^\sharp = 1$. In particular, $j+t\in \mathscr N$. If
$\mathscr N\nleq U^*$ then for an arbitrary element $n\in (\mathscr N\cap
\mathscr W)\setminus U^*$, $w+n \in \mathscr K$. This means that the element
$\alpha(w)$ has two $\mathscr H^\sharp \mathscr K^\sharp$ decompositions:
\[\alpha(w)=\alpha(w+n)+0=\alpha(s+j)+\alpha(j+t).\]
This is only possible if $j+t\in \mathscr N$, thus our claim is proved.

Let $M'$ be the $S_5$-submodule generated by $j+t$, then $J_0 \leq M'\leq
\mathscr N$ as the irreducible $J_0$ is not $S_5$-isomorphic to a submodule of
$U^*$. This contradiction proves the simplicity of $\mathscr Q$. 
\end{proof}

\begin{remark}
We have seen that there are at least two possibilities for the choice of
$\mathscr H$. Also in $\widetilde{\mathscr W}$, we can replace any minimal
submodule $W^{**}$ by an appropriate $S^{**}$. This shows
that there are \textit{many} Bol loops of exponent $2$ which live in the same
non-solvable group. Many of these loops are simple. Using computer
calculations, we were able to construct over $30$ nonisomorphic simple Bol loops
of exponent $2$ in $\mathscr G$ in the case $k=1$. 
\end{remark}

In fact, this phenomena in not unusual for Bol loops of exponent $2$. In
\cite[Section 5]{KiechleNagyG} and \cite[Theorem 5.5]{NagyGfrattini}, the
authors constructed rich classes of Bol loops of exponent $2$ having the same
enveloping groups, namely the wreath product $C_2^n \wr C_2$ and the
extraspecial $2$-group $E_{2^{2n+1}}^+$, respectively. In these cases, a simple
parametrization of the conjugacy classes of involutions enabled a description of
the associated loops. Unfortunately, the group $\mathscr G$ has many conjugacy
classes of involutions and these classes have no nice algebraic
parametrization. Therefore, we see no way of classifying all simple Bol loops
with enveloping group $\mathscr G$. 

The above remark lets us make another observation. While the class of finite Bol
loops of exponent $2$ is very rich, the structure of the right multiplication
group of a Bol loop of exponent $2$ is rather restricted. Differently speaking,
while the classification of finite simple Bol loops of exponent $2$ seems to be
hopeless, we think that the classification of right multiplication groups of
such loops could be a meaningful project.

\medskip 

We finish this paper with the following 

\begin{question}
Classify those almost simple groups $T$ for which an exponent $2$ Bol loop
folder $(G,H,K)$ exists such that $T\cong G/O_2(G)$.
\end{question}

\bibliographystyle{plain}

\end{document}